\documentclass{amsart}
\usepackage[utf8]{inputenc}
\usepackage{graphicx} 
\usepackage{amssymb}
\usepackage{verbatim}
\usepackage{hyperref}
\usepackage{mathtools}
\usepackage{bbm}
\usepackage{enumerate}

\usepackage{xy}
\xyoption{all}

\usepackage{tikz}

\allowdisplaybreaks 

\usepackage{setspace}
\usepackage[
backend=biber,
style=numeric,
]{biblatex}
\usepackage{color}
\newcommand{\red}[1]{{\color{red}{#1}}}

\usepackage{comment}

\newcommand{\onto}{\twoheadrightarrow}

\newcommand{\Z}{\mathbb{Z}}
\newcommand{\Q}{\mathbb{Q}}

\newcommand{\C}{\mathbb{C}}
\newcommand{\unit}{\mathbbm{1}} 

\newcommand{\CC}{\mathcal{C}}
\newcommand{\DD}{\mathcal{D}}
\newcommand{\EE}{\mathcal{E}}
\newcommand{\FF}{\mathcal{F}}
\newcommand{\GG}{\mathcal{G}}

\newcommand{\JJ}{\mathcal{J}}
\newcommand{\KK}{\mathcal{K}}

\newcommand{\PP}{\mathcal{P}}
\newcommand{\VV}{\mathcal{V}}
\newcommand{\Neg}{\mathcal{N}}

\newcommand{\tr}{tr}
\newcommand{\id}{id}
\newcommand{\ldag}{\langle}
\newcommand{\rdag}{\rangle^\dagger}
\newcommand{\ol}{\overline}

\newcommand{\dd}{\mathfrak{d}}
\newcommand{\ee}{\mathfrak{e}}
\renewcommand{\gg}{\mathfrak{g}}
\newcommand{\hh}{\mathfrak{h}}
\renewcommand{\sl}{\mathfrak{sl}}
\newcommand{\so}{\mathfrak{so}}
\newcommand{\cat}[1]{\ol{\Rep(U_{#1}(\gg_2))}}

\DeclareMathOperator{\Rep}{Rep}
\DeclareMathOperator{\Hom}{Hom}

\DeclareMathOperator{\GPA}{GPA}
\DeclareMathOperator{\End}{End}
\DeclareMathOperator{\Vecc}{Vec}

\DeclareMathOperator{\Ab}{Ab}
\DeclareMathOperator{\Idemp}{Idemp}
\DeclareMathOperator{\Add}{Add}
\DeclareMathOperator{\Right}{Right}
\DeclareMathOperator{\Dec}{Dec}

\renewcommand{\phi}{\varphi}
\let\epsilon\varepsilon 

\newcommand{\blank}{\rule{0.2cm}{0.15mm}}

\theoremstyle{plain}
    \newtheorem{definition}{Definition}
    
    \newtheorem{proposition}{Proposition}
    \newtheorem{theorem}{Theorem}
    \newtheorem{lemma}{Lemma}
    \newtheorem{corollary}{Corollary}
    
    \newtheorem{remark}{Remark}

\newcommand{\skein}[2]{\raisebox{-.4\height}{ \includegraphics[scale = #2]{figs/#1.png}}}

\title{Type $G_2$ Quantum Subgroups from Graph Planar Algebra Embeddings}
\author{Caleb Kennedy Hill}
\address{Caleb Kennedy Hill\\
University of New Hampshire\\
Durham, 
New Hampshire}
\email{caleb.hill@unh.edu}
\date{}

\begin{document}

\begin{abstract}
    We give graphical presentations for the two quantum subgroups of type $G_2$.
    To do this we use a method of extending a tensor category by embedding the
    planar algebra of a $\otimes$-generating object into the graph planar algebra
    of this object's fundamental graph.
    This allows the use of computational methods to uncover relations 
    we would have little hope of arriving at otherwise.
\end{abstract}
\maketitle


\section{Introduction}\label{sec:intro}
Quantum subgroups are a well-known source of tensor categories.
More precisely, given a conformal embedding $\VV(\gg,k) \subseteq \VV(\hh,1)$ of VOAs 
as in \cite{DMNO}, one obtains a corresponding etale algebra $A$ in $\ol{\Rep(U_{q}(\gg))}$,
where $q$ depends on $k$..
This algebra then allows one to consider the category $\ol{\Rep(U_{q}(\gg))}_A$ of right $A$-modules
Commutativity of $A$ then gives a tensor product on $\ol{\Rep(U_{q}(\gg))}_A$, and one may study this
new category in its own right.
The free functor gives an embedding $\ol{\Rep(U_{q}(\gg))} \hookrightarrow \ol{\Rep(U_{q}(\gg))}_A$
which is, in general, not full.
Hence it remains to find a description of the new morphisms in $\ol{\Rep(U_{q}(\gg))}_A$ 
to describe this newly constructed category of modules.
Recent works of Edie-Michell and Snyder \cite{cain_noah} have used this reasoning, and 
representation theoretic techniques to give diagrammatic descriptions of tensor categories of modules
corresponding to the family of conformal embeddings $\VV(\sl_N,N) \subseteq \VV(\so_{N^2-1},1)$.

On the other hand, one may start with a known category $\CC$ and compute an embedding 
\[
    \CC \hookrightarrow \DD
\]
for $\DD$ some category in which explicit calculations are more easily performed by computers.
One option for $\DD$ is the graph planar algebra $\GPA(\Gamma)$ for some graph $\Gamma$.
By the GPA Embedding Theorem \cite{extended_haagerup}, an embedding
\[
    \CC \hookrightarrow \GPA(\Gamma)
\]
induces a module category for $\CC$.
This has been done for $\ol{\Rep(U_q(\sl_N))}$ in \cite{Cain_Dan} and for the 
extended Haagerup categories in \cite{extended_haagerup}.

The present work describes a blend of these two techniques, based on the two conformal embeddings
\begin{equation}\label{eq:conf-embs}
    \VV(\gg_2,3) \subseteq \VV(\ee_6,1) \quad\text{and}\quad \VV(\gg_2,4) \subseteq \VV(\dd_7,1).
\end{equation}
We begin by finding an embedding of the well-known trivalent category $\GG_2(q)$ of \cite{Kuperberg,tricats}
into the graph planar algebra on the module fusion graph of the object $V_{\Lambda_1}$
in $\ol{\Rep(U_q(\gg_2))}_A$.
The category $\GG_2(q)$ is known to be a diagrammatic presentation for $\ol{\Rep(U_q(\gg_2))}$.
Through the free functor we can view our GPA embedding as a GPA embedding for a $\otimes$-generating 
object's planar algebra in $\ol{\Rep(U_q(\gg_2))}_A$.
Diagrammatically, this gives us a black-strand and a trivalent vertex:
\[
    \skein{/skein_figs/trivalent}{0.14}
\]
with a known skein theory.
We then search inside the GPA embedding for new morphisms.
According to \cite{DMNO} there ought to be a projection onto a $\Z_k$-like simple object in $\ol{\Rep(U_q(\gg_2))}_A$,
so this is what we search for inside the GPA.
We view this new morphism as an I with an oriented red vertical strand:
\[
    \skein{/skein_figs/dec_tet_5}{0.15}       
\]
In the case of $\GG_2(q)$, the properties of this new morphism were unknown beyond a few basics deriving from, e.g.,
the fact that is is a projection onto a simple object contained in the tensor square of the $\otimes$-generating object.
Once we have our hands on the image in the GPA of this projection, though, 
we may explore its properties through explicit computations.
We perform this process of extending GPA embeddings for the conformal embeddings
\begin{equation*}
    \VV(\gg_2,3) \subseteq \VV(\ee_6,1) \quad\text{and}\quad \VV(\gg_2,4) \subseteq \VV(\dd_7,1)
\end{equation*}
mentioned above.
By \cite{gannon2025_exotic_q_subgroups_extensions} these conformal embeddings correspond to 
the only etale algebra objects in $\ol{\Rep(U_q(\gg_2))}$.

Now we begin by introducing some notation for a skein theory involving an oriented, red strand 
in addition to unoriented black strands.
A diagram which consists of black unoriented strands as well as red oriented strands will be
referred to as a {\bf decoration} of the diagram given by dropping all red strands;
such a diagram is called {\bf decorated}.
\begin{definition}
    For a diagram $\EE$ the notation $\Right^i(\EE)$ means an $i$-click right rotation. 
    For instance, 
    \[
    \Right^1\left( \skein{/skein_figs/dec_tet_4}{0.1} \right) = \skein{/skein_figs/dec_tet_3}{0.1} 
    \quad\text{and}\quad 
    \Right^2\left( \skein{/skein_figs/dec_tet_4}{0.1} \right) = \skein{/skein_figs/dec_tet_4}{0.1}.
    \] 

    Suppose the diagram $\EE$ has $m$ boundary points. 
    We define $\Dec_i(\EE)$ to be the $i$-th external single clockwise decoration of $\EE$. 
    For example,
    \[
        \Dec_1\left(\skein{/skein_figs/trivalent}{0.1}\right) = \skein{/skein_figs/triv_rightDown}{0.1},
    \quad
    \text{and} 
    \quad
        \sum_{i=1}^{3} i\Dec_i\left(\skein{/skein_figs/trivalent}{0.1}\right) 
        = \skein{/skein_figs/triv_rightDown}{0.1} 
        +2\skein{/skein_figs/triv_bottomLeft}{0.1} 
        +3\skein{/skein_figs/triv_leftUp}{0.1}
    \]
    We adopt the convention that $\Dec_0 (\EE) = \EE$.
\end{definition}

Both of the categories studied in this paper are extensions of trivalent categories by a red, directed, $\Z_n$-like strand. 
Here we define the class of categories we will be working with.
Later we will show that, with a relatively tame assumption on the underlying skein theory, 
categories in this class are evaluable in general.
\begin{definition}\label{def:zn-ext}
    Let $\CC = \left\langle \skein{/skein_figs/trivalent}{0.1} \right\rangle$ 
    be a trivalent category. 
    Call $\DD$ a {\bf $\Z_n$-like extension} (or {\bf cyclic} when $n$ is understood) of $\CC$ if we have
    $\DD = \left\langle \skein{/skein_figs/trivalent}{0.1}, \skein{/skein_figs/dec_tet_5}{0.1} \right\rangle$,
    enjoying the following relations 
    \begin{equation*}\tag{Recouple}
        \skein{/skein_figs/recouple_LHS}{0.1} 
        = \skein{/skein_figs/recouple_RHS}{0.1}
    \end{equation*}

    \begin{equation*}\tag{Reverse}
        \skein{/skein_figs/down_n-1}{0.1} 
        = \skein{/skein_figs/dec_tet_5}{0.15}
    \end{equation*}

    \begin{equation*}\tag{Schur 0}
        \skein{/skein_figs/schur_0_1}{0.1} = 0 \quad 
        \skein{/skein_figs/schur_0_2}{0.1} = 0 \quad \cdots \quad 
        \underbrace{ \skein{/skein_figs/schur_0_n-1}{0.11} }_{\text{$n-1$}} = 0
    \end{equation*}

    \begin{equation*}\tag{Schur 1}
        \skein{/skein_figs/schur_1_1}{0.1} = 0 \quad 
        \skein{/skein_figs/schur_1_2}{0.1} = 0 \quad \cdots \quad 
        \underbrace{ \skein{/skein_figs/schur_1_n-1}{0.11} }_{\text{$n-1$}} = 0
    \end{equation*}

    \begin{equation*}\tag{Swap}
        \skein{/skein_figs/swap_LHS}{0.15} 
        = \omega \skein{/skein_figs/swap_RHS}{0.15}
    \end{equation*}

    \begin{equation*}\tag{decStick}
        \skein{/skein_figs/stick_w_Pg}{0.25} 
        = \skein{/skein_figs/stick}{0.1}
    \end{equation*} 

    \begin{equation*}\tag{decBigon}
        \skein{/skein_figs/g_bigon2}{0.1} = b \skein{/skein_figs/stick}{0.15}
    \end{equation*}

    \begin{equation*}\tag{Change of Basis}
        \skein{/skein_figs/triv_leftUp}{0.1} 
        = \sum_{i=0}^{n-1} r_i \skein{/skein_figs/triv_rightUp_i}{0.1}
    \end{equation*}

    \begin{equation*}\tag{decTrigon}
        \skein{/skein_figs/dec_trigon_LHS}{0.1} 
        = \sum_{i=0}^{n-1} t_i \skein{/skein_figs/triv_rightUp_i}{0.1}
    \end{equation*}\label{eq:decTrigon}
 
    \begin{equation*}\tag{decTetragon}
         \skein{/skein_figs/dec_tet_LHS}{0.1} 
         = \sum_{i=0}^4 \sum_{j=0}^3 u_{i,j} \Dec_i\left( \Right^j\left( \skein{/skein_figs/dec_tet_1}{0.1} \right) \right)
         + \sum_{i=0}^4 \sum_{j=0}^3 v_{i,j} \Dec_i\left( \Right^j\left( \skein{/skein_figs/dec_tet_4}{0.1} \right) \right)
    \end{equation*}

     \begin{equation*}\tag{decPentagon}
        \skein{/skein_figs/dec_pent_LHS}{0.1} 
        = \sum_{i=0}^5 \sum_{j=0}^4 w_{i,j} \Dec_i\left( \Right^j\left( \skein{/skein_figs/dec_pent_RHS1}{0.1} \right) \right) 
        + \sum_{i=0}^5 \sum_{j=0}^4 x_{i,j} \Dec_i\left( \Right^j\left( \skein{/skein_figs/dec_pent_RHS2}{0.1} \right) \right)
    \end{equation*}
    for $c, \omega, b, r_i, t_i, u_{i,j}, v_{i,j}, w_{i,j}, x_{i,j} \in \C$.
    We additionally enforce the condition that the diagrams $\skein{/skein_figs/triv_rightUp_i}{0.075}$ 
    for $i=0,\dots,n-1$ span the $2\to1$ hom-space, where we adopt the convention that 
    an oriented strand labelled with $i$ refers to the presence of $i$
    parallel oriented strands.
\end{definition}

\begin{definition}\label{def:D4}
    Set $q_4 = e^{\frac{2\pi i}{48}}$ and define $\DD_4$ to be the 
    cyclic extension of $\GG_2(q_4)$ with the following structure constants:\footnote{
        We omit the (decPentagon) equation here for brevity. 
        It contains 44 nonzero summands, and can be found on the author's GitHub.
    } 
    
    \begin{equation*}
        \skein{/skein_figs/swap_LHS}{0.1} = - \skein{/skein_figs/swap_RHS}{0.1}, \qquad
        \skein{/skein_figs/g_bigon2}{0.1} = (q_4+1+q_4^{-1}) \skein{/skein_figs/stick}{0.1}
    \end{equation*}

    \begin{equation*}
        \skein{/skein_figs/triv_leftUp}{0.1} =  
        q_4^{-4} \skein{/skein_figs/trivalent}{0.1} 
        + q_4^{16} \skein{/skein_figs/triv_rightUp}{0.1}
    \end{equation*}

    \begin{equation*}
        \skein{/skein_figs/dec_trigon_LHS}{0.1} = 
        -\skein{/skein_figs/trivalent}{0.1} 
        -\skein{/skein_figs/triv_rightUp}{0.1}
    \end{equation*}

    \begin{align*}
         \skein{/skein_figs/dec_tet_LHS}{0.12} & = q^2 \skein{/skein_figs/dec_tet_1}{0.1} + q_4^2 \skein{/skein_figs/dec_tet_2}{0.1} 
            + \frac{q_4^{17}}{q-q^{-1}} \skein{/skein_figs/dec_tet_3}{0.1} + q_4^2 \skein{/skein_figs/dec_tet_4}{0.1} \\
        & + \frac{1+[3]_{q_4}}{q_4^4} \skein{/skein_figs/dec_tet_5}{0.1} + \frac{[2]_{q_4}}{q_4^{13}} \skein{/skein_figs/dec_tet_6}{0.1}
            + q_4^{-14} \skein{/skein_figs/dec_tet_8}{0.1} + \frac{[2]_{q_4}}{q_4^{13}} \skein{/skein_figs/dec_tet_9}{0.1}
            + (-1) \skein{/skein_figs/dec_tet_10}{0.1}
    \end{align*}

\end{definition}

One of the two primary results we give here is that $\DD_4$ is a presentation for the category of modules corresponding to the
level 4 conformal embedding of $\gg_2$.
\begin{theorem}
    There is an equivalence
    \[
        \Ab(\ol{\DD_4}) \cong \ol{\Rep(U_{q_4}(\gg_2))}_{A_4}
    \]
    where $A_4$ is the algebra object corresponding to the level-4 conformal embedding given in Equation~\ref{eq:conf-embs}.
\end{theorem}
Theorem~\ref{thm:level-3} is an analogous theorem for level 3, where $\DD_3$ is defined similarly to $\DD_4$,
with structure constants given in the attached Mathematica files.

It is not clear a priori that the defining relations for, say, $\DD_4$ lead to a nontrivial tensor category.
The general undecidability of the word problem for groups offers some evidence that this question is difficult
for a typical presentation for a tensor category.
That is, one should not expect a set of relations to yield any nontriviality.
The fact that we have a nonzero GPA embedding of $\DD_4$ is what tells us that $\DD_4$ itself is nontrivial.

The remainder of the paper is structured as follows.
Section~\ref{sec:prelim} sets up most of the theory needed, referencing that which we do not exposit here.
This includes unoriented planar algebras, unoriented graph planar algebras, 
internal algebra and module objects, and some assorted theoretical devices and results.
Section~\ref{sec:skein} then goes on to investigate some general properties of cyclic extensions.
We expect this class of categories to be of use for researchers intent on conjuring 
examples of exotic tensor categories.
In fact, in a forthcoming paper, the present author and Edie-Michell will diagrammatically
present a number of near-group categories as cyclic extensions of $SO(3)_q$ trivalent categories.
We demonstrate evaluability of this class of categories under a relatively tame assumption on the
underlying trivalent skein theory.
Section~\ref{sec:gpa-emb} discusses the process of arriving at GPA embeddings.
We detail the techniques used to arrive at GPA embeddings of trivalent categories, and then
show how we extend these embeddings to cyclic extensions.
Finally, in Section~\ref{sec:equivalences} we prove that the representations we've found 
are actually presentations for the respective quantum subgroups.

Throughout the remainder of this paper we reference several Mathematica notebooks.
These can be found at the author's GitHub:
\begin{center}
    \href{https://github.com/calebkennedyhill/quantum_subgroups_G2}{https://github.com/calebkennedyhill/quantum\_subgroups\_G2}.
\end{center}
Dependencies and instructions to recreate the verifications can be found in the file \verb|README.md|
and in the notebook \verb|/code/level-3/nb-1_setup-1.nb| for level 3 and the notebook 
\verb|/code/level-4/nb-1_setup-1.nb| for level 4.
All formal verification of all equations were finished at commit \verb|0cc217e|.

\subsection*{Acknowledgements}

The author would like to thank Cain Edie-Michell for the idea to pursue this project,
and for the innumerable helpful conversations along the way.
The University of New Hampshire provided financial support in the form of several 
fellowships and grants during the preparation of this manuscript.

\section{Preliminaries}\label{sec:prelim}
Here we define the tools we'll use. 
This includes planar algebras, graph planar algebras, and internal algebra and module objects.
We give only a few necessary results, and refer the reader to the definitive publications.
For the general theory of tensor categories, see \cite{EGNO}.

\subsection{Algebra and Module Objects}
We will ultimately show that $\DD_3$ and $\DD_4$ are presentations for the categories
$\cat{q_3}_{A_3}$ and $\cat{q_4}_{A_4}$ of modules over algebra objects $A_3$ and $A_4$ coming from the conformal embeddings 
$\VV(\gg_2,3) \subseteq \VV(\ee_6,1)$ and $\VV(\gg_2,4) \subseteq \VV(\dd_7,1)$, respectively. 
In this subsection we recall basic facts about algebra and module objects,
as well as conformal embeddings. 
See \cite{EGNO, ostrik2001modulecategoriesweakhopf} for more complete descriptions. 
The theory which will apply to our context is given in \cite{cain_noah}.
Some basic properties concerning the interaction of algebra and module objects with
monoidal functors will be used in the proof of our main theorems;
this material can be found in \cite{monoidalFunctorsAndAlgebras}.
We restate a few definitions and facts here.
Unless otherwise stated, we will be assuming the underlying tensor categories are braided.

\begin{definition}
    Let $A$ be an algebra object of the braided tensor category $\CC$.
    $A$ is an {\bf etale} algebra if it is commutative and separable.
    We call $A$ {\bf connected} if it is etale and $\dim\Hom_\CC(\unit \to A)=1$.
\end{definition}

For an etale algebra object $A$ of $\CC$, we denote by $\CC_A$ the collection of right $A$-modules internal to $\CC$.
As described in \cite{cain_noah}, commutativity of $A$ induces a tensor product on $\CC_A$.
Separability of $A$ implies semisimplicity of $\CC_A$, 
and connectedness of $A$ implies the unit $\unit_{\CC_A}=A$ is simple in $\CC_A$ \cite{DMNO}.

The etale conditions are precisely those required to perform the skein theory on $A$-modules 
to define the tensor product on $\CC_A$.
etale is the same thing as multiplication having a section, which means we can pop $A$-bigons.
Connected means theres only one $A$-cap and one $A$-cup.
This lets us wiggle enough to perform the proof that $M\otimes_A N$ is well-defined.

Furthermore, when $\CC$ is semisimple, the free functor 
\[
    \FF_A: \CC \xhookrightarrow{X\mapsto A\otimes X} \CC_A
\]
is faithful monoidal and, as we will see later, not always full.
Its right adjoint is the forgetful functor $\FF^\vee:\CC_A\to\CC$ 
which acts on objects $(M,\mu_M)$ by dropping the multiplication map $\mu_M: M\otimes A\to M$
and on morphisms as the identity.

Beginning with the following lemma, which is recreated from \cite{exactSequencesTensorCategories},
we now recall some facts that will help us along the way.
These will consist of a few results, along with the basics of {\it conformal embeddings}.

\begin{lemma}\label{lem:exact-functor}
    Let $\FF:\CC\to\DD$ be a monoidal functor with faithful exact right adjoint $R$. 
    If we define $A\coloneqq R(\unit)$, then there is an equivalence $\KK$ such that the diagram
    \[
    \xymatrix@C=60pt@R=45pt{
    \CC \ar[r]^{\FF} \ar[dr]^{\FF_A} & \DD \ar[d]^{\KK} \\
     & \CC_A
     }
    \]
    commutes up to natural isomorphism.
\end{lemma}

\begin{lemma}\label{lem:simple-unit-unitary}
    Suppose $\CC$ has simple unit, $\DD$ is unitary, and $\FF:\CC\to \DD$ is a $\dagger$-functor.
    Then $\ol{\CC}$ is unitary, and $F$ descends to a $\dagger$-embedding $\ol{\FF}:\ol{\CC}\to \DD$ such that
    \[
    \xymatrix@C=60pt@R=45pt{
    \CC \ar[r]^{\FF} \ar@{->>}[d]^{} & \DD \\
    \ol{\CC} \ar@{^{(}->}[ur]^{\ol{\FF}} & 
     }
    \]
    commutes.
\end{lemma}

One result which will help immensely in arriving at GPA embeddings is the following,
which is Lemma 2.4 of \cite{cain_noah}.
Since the free functor $\FF_A$ gives an embedding $\CC \hookrightarrow \CC_A$, 
the braid $c_{X,Y}$ in $\CC$ is mapped to $\FF_A(c_{X,Y})$ which defines a braided structure
for the subcategory $\FF_A(\CC)$ of $\CC_A$.
Since the free functor is in general not full, we cannot extend this to a braiding on all of $\CC_A$,
however, there is a half-braid for arbitrary morphisms between objects in the image of $\FF_A$.

\begin{lemma}[Half-braid]\label{lem:general-half-braid}
    Let $\CC$ be a braided tensor category, and $A$ an etale algebra object.
    For any $f\in\Hom_{\CC_A}( \FF_A(Y_1) \to \FF_A(Y_2) )$, the following relation holds:
    \begin{equation*}\tag{Half-braid}
        \skein{/skein_figs/arb_hb_top}{0.2} = \skein{/skein_figs/arb_hb_bottom}{0.2}
    \end{equation*}
\end{lemma}

Note that $f$ need not be in the image of the free functor.
We will utilize this result to obtain a rather large number (2970 at level 3 and 7776 at level 4)
of linear equations constraining the GPA coordinates
of the morphisms not living in the image of $\FF_A$.
Thus the half-braid relation will be key to our program, despite not being necessary to prove evaluability.

The source of our algebra objects will be conformal embeddings.
We direct the reader to \cite{DMNO} a more complete treatment.
For a vertex operator algebra $\VV(\gg,j)$, define $\CC(\gg,j) \coloneq \Rep(\VV(\gg,j))$.

Affine Lie algebras and conformal embeddings will only be used to obtain algebra objects 
and module fusion graphs, so we briefly recall the correspondence 
\begin{equation}\label{eq:G2-affine-quantum}
    \CC(\gg_2,k) \cong \ol{\Rep(U_{q_k}(\gg_2))} 
\end{equation}
of \cite{finkelberg1996}, where $k$ is the level and $q_k$ is given by
\[
    q_k = e^{\frac{2\pi i}{3(4+k)}}.
\]

From \cite[Appendix]{DMNO} we recall the conformal embeddings which are of use to us:
\begin{equation*}
    \VV(\gg_2,3) \subseteq \VV(\ee_6,1) \quad\text{and}\quad \VV(\gg_2,4) \subseteq \VV(\dd_7,1).
\end{equation*}
At level 3 we have $q_3 = e^{\frac{2\pi i}{42}}$ and at level 4 we have $q_4 = e^{\frac{2\pi i}{48}}$.
We obtain the algebra objects and fundamental graphs for GPAs from \cite{g2_graphs}:
\begin{equation}\label{eq:alg-objetcs}
    A_3 = V_{\emptyset} \oplus V_{\Lambda_1} \quad\text{and}\quad A_4 = V_{\emptyset} \oplus V_{3\Lambda_1}
\end{equation}
at levels 3 and 4, respectively.

Additionally at level 3 and 4, respectively, we have the existence of 
$\Z_3$-like and $\Z_2$-like simple objects $g_3$ and $g_4$ \cite{DMNO}.
From \cite{g2_graphs}
we see that at both levels $k=3,4$ we have
\[
    \dim\Hom_{ \ol{\Rep(U_{q_k}(\gg_2))}_{A_k} }(\FF_{A_k}(V_{\Lambda_1})^{\otimes2} \to g_k) = 1.
\]

\begin{remark}\label{rem:Pg-properties}
    It follows that there are idempotents 
    \[
        P_{g_k} : \FF_{A_k}(V_{\Lambda_1})^{\otimes2} \to \FF_{A_k}(V_{\Lambda_1})^{\otimes2}
    \]       
    projecting onto these grouplike objects.
    As the $g_i$ are simple, we have $P_{g_k}^\dagger = P_{g_k}$.

    The behavior of the $P_{g_k}$ will be captured by $\skein{/skein_figs/dec_tet_5}{0.1}$ in $\DD_k$.
    As we will see, describing the interaction of $P_{g_k}$ with the image of the free functor 
    will be sufficient to fully describe $\ol{\Rep(U_{q_k}(\gg_2))}_{A_k}$.
\end{remark}

\subsection{Unoriented Planar Algebras}
Recall the theory of {\bf rigid} monoidal categories detailed in \cite{KW}. 
To put it succinctly, rigid monoidal categories have duals. 
Duals, and the associated evaluation and coevaluation maps, give us cups and caps. 
We also assume pivotality throughout, which gives us the ability to isotope diagrams. 
The generators we will use for our planar algebras will be symmetrically self-dual.

Let $X$ be a (symmetrically self-dual) {\bf tensor generator} for the tensor category $\CC$; 
that is, every object of $\CC$ is isomorphic to a subobject of some tensor power $X^{\otimes n}$. 
Let $\PP_{X;\CC}$ be the full subcategory of $\CC$ whose objects are tensor powers $\unit=X^{\otimes0},X,X^{\otimes2},\dots$; 
we call this the (unoriented) {\bf planar algebra} generated by $X$ in $\CC$. 
The planar algebra $\PP_{X;\CC}$ is {\bf evaluable} if $\dim\End_{\PP_{X;\CC}}(\unit)=1$. 

We will be presenting the our two quantum subgroups as extensions of $\GG_2(q)$ skein theories, 
in the spirit of Kuperberg \cite{Kuperberg,tricats}. 
Up to a rescaling by a factor of $\kappa = \sqrt{[7]-1}$ we use the same skein theory as \cite{tricats}
(note the sign error in the \ref{eq:Pentagon} relation of \cite{Kuperberg}). 
\begin{definition}
    For $q$ a root of unity, the $\GG_2(q)$ skein theory is defined to be that generated by an 
    unoriented trivalent vertex $\skein{/skein_figs/trivalent}{0.05}$ satisfying the relations
    \begin{equation*}\tag{Loop}
        \skein{/skein_figs/loop}{0.1} = \delta = q^{10} + q^{8} + q^{2} + 1 + q^{-2} + q^{-8} + q^{-10}
    \end{equation*}

    \begin{equation*}\tag{Lollipop}\label{eq:Lolli}
        \skein{/skein_figs/lollipop}{0.1} = 0  
    \end{equation*}

    \begin{equation*}\tag{Rotate}\label{eq:Rotate}
        \Right^1\left( \skein{/skein_figs/trivalent}{0.1} \right) = \skein{/skein_figs/trivalent}{0.1}
    \end{equation*}

    \begin{equation*}\tag{Bigon}\label{eq:Bigon}
        \skein{/skein_figs/bigon}{0.1} = \kappa^2 \skein{/skein_figs/stick}{0.1}
    \end{equation*}

    \begin{equation*}\tag{Trigon}\label{eq:Trigon}
        \skein{/skein_figs/trigon_LHS}{0.1} = -(q^4 +1+ q^{-4}) \skein{/skein_figs/trivalent}{0.1}
    \end{equation*}

    \begin{equation*}\tag{Tetragon}\label{eq:Tetragon}
        \skein{/skein_figs/tet_LHS}{0.1} 
        = (q^2 + q^{-2}) \left( \skein{/skein_figs/dec_tet_1}{0.1} 
        + \skein{/skein_figs/dec_tet_2}{0.1} \right) 
        + (q^2 +1+ q^{-2}) \left( \skein{/skein_figs/dec_tet_3}{0.1} 
        + \skein{/skein_figs/dec_tet_4}{0.1} \right)
    \end{equation*}

    \begin{equation*}\tag{Pentagon}\label{eq:Pentagon}
        \skein{/skein_figs/pentagon}{0.1} = 
        - \sum_{i=0}^4 \Right^i \left( \skein{/skein_figs/dec_pent_RHS1}{0.1} \right) 
        - \sum_{i=0}^4 \Right^i \left( \skein{/skein_figs/dec_pent_RHS2}{0.1} \right)
    \end{equation*}
\end{definition}

Our use of planar algebras will depend entirely on the construction of the Cauchy completion,
which we sketch here.
See \cite{cain_noah} for more details and \cite{tuba_wenzl} for a full treatment of the topic.
Recall that the {\bf idempotent completion} of a pivotal tensor category $\CC$ consists of pairs $(Z,p)$,
where $p\in\End_\CC(Z)$ is an idempotent.
We denote the idempotent completion of $\CC$ as $\Idemp(\CC)$.
Further, we define the {\bf additive envelope} of a pivotal, $\C$-linear tensor category $\CC$
to have objects formal direct sums $\bigoplus_j Z_j$ for objects $Z_j$ of $\CC$.
The {\bf Cauchy completion} of $\CC$ is defined by 
\[
    \Ab(\CC) \coloneq \Add(\Idemp(\CC)).
\]

If we again assume $X$ tensor generates $\CC$, it follows that $\CC \cong \Ab(\PP_{X;\CC})$ \cite[Theorem 3.4]{tuba_wenzl}.
The universal property of $\Ab(\PP_{X;\CC})$ therefore implies 
that studying $\PP_{X;\CC}$ is sufficient to understand $\CC$. 

By \cite[Corollary 2.20]{bodish_triple_clasp_g2} and \cite[Corollary 6.7]{reconstructing_g2}, the category $\GG_2(q)$ is a {\bf presentation} for the category $\cat{q}$ in the sense that
\[
\ol{\Rep(U_q(\gg_2))} \cong \Ab(\ol{\GG_2(q)}).
\]

\subsection{Unoriented Graph Planar Algebras}
We will study the quantum subgroups of type $G_2$ by embedding their skein theories into appropriate graph planar algebras (GPAs). 
This serves two purposes:
\begin{itemize}
    \item Giving us solid ground on which to do computations, allowing us to uncover relations by finding them in the GPA hom-spaces, and
    \item Implying unitarity for the skein theories,
\end{itemize}
GPAs are an invention of Vaughan Jones \cite{jones_GPA}.
In this work we have no use for more general GPAs so we consider only the unoriented case. 

\begin{definition}\label{def:GPA}
    Let $\Gamma=(V,E)$ be a finite graph. 
    For an edge $e=(u,v)\in E$, let $\ol{e} \coloneq (v,u)\in E$. 
    The {\bf graph planar algebra} on $\Gamma$, denoted $\GPA(\Gamma)$, is the strictly pivotal rigid monoidal category 
    whose objects are nonnegative integers, and whose hom-spaces have basis
    \[
        \Hom_{\GPA(\Gamma)}(m\to n) \coloneqq \C\left\{ (p,q) \mid \substack{\text{$p$ an $m$-path} \\ \text{$q$ and $n$-path}}, \substack{\text{$s(p)=s(q)$}\\ \text{$t(p)=t(q)$}} \right\},
    \]
    with composition law\footnote{With the convention $f\circ g \coloneq f(g)$.}
    \[
        (p,q)\circ(p',q')\coloneqq \delta_{q=p'} (p,q'),
    \]
    and rigidity maps
    \[
        ev = \sum_e \sqrt{ \frac{\lambda_{t(e)}}{\lambda_{s(e)}} } \langle e\ol{e},s(e) \rangle, 
        \quad coev = \sum_e \sqrt{ \frac{\lambda_{t(e)}}{\lambda_{s(e)}} } \langle s(e) e\ol{e} \rangle
    \]
    where $\lambda$ is the Frobenius-Perron eigenvector of the adjacency matrix of $\Gamma$.
    The monoidal product on objects is addition, and for morphisms is defined by 
    \[
        (p,q)\otimes(p',q')\coloneqq \delta_{s(p')=t(p)} (pp',qq').
    \]
\end{definition}

\section{Cyclic extensions}\label{sec:skein}

The goal of this section is to develop the tools needed to prove evaluability of general $\Z_n$-like extensions of trivalent categories.
We expect this class of extensions to be helpful in the search for novel categories.
For example, there is work underway by the present author and Edie-Michell to use the techniques of this paper to 
construct a class of examples of {\it near-group} categories, as defined in \cite{gannon_near-groups}.
This work on near-group categories extends an underlying $SO(3)_q$ trivalent skein theory.
The present author has also computed extensions for two categories of type $SP(4)_q$, which, despite its skein theory being
generated by a braid, is of the same essence. 

This all begs the question of which leaves on the ``tree of life'' of \cite{tricats} 
might bear more fruit of this variety.
Already we have extended both categories ($SO(3)_q$ and $Fib$) covered by \cite[Theorem A]{tricats} by group-like objects.
This paper deals with all but one of the categories covered by \cite[Theorem B]{tricats}.
The categories one might next attempt such an extension of include:
\begin{itemize}
    \item The remaining category $ABA$ of \cite[Theorem B]{tricats}
    \item The category $H_3$ of \cite[Theorem C]{tricats}
\end{itemize}

General methods for demonstrating evaluability of a skein theory involve identifying 
some measure of complexity for a closed diagram, then showing the known 
relations allow one to strictly decrease this measure. 
For our underlying trivalent categories, {\it Euler-evaluability} allows us to 
decrement one measure of complexity: number of internal faces. 
With the new strand type, we have another measure: number of red strands. 
The underlying trivalent categories we deal with have evaluation algorithms based on the standard Euler characteristic argument.
One way to capture this evaluability is by considering dimensions of box spaces.
\begin{definition}
    In a trivalent category we define a {\bf box space} $B(k,f)$ to be the span of diagrams $k\to0$ with $f$ internal faces.
    If $\CC$ is a trivalent category such that, for $k=1,\dots,5$, the containment
    \[
    B(k,1) \subseteq B(k,0)
    \]
    holds, we will refer to $\CC$ as {\bf Euler-evaluable}.
\end{definition}

Diagrams inside a $\Z_n$-like extension exhibit the following nice properties, which will be key in proving their evaluability.
Essentially, we use the following lemmas to exchange decorated faces for singly-externally-decorated faces.
The defining relations for a $\Z_n$-like extension then pop the singly-decorated faces.

\begin{lemma}
    (1) ($\Z_n$) follows from (Recouple) and (Reverse).

    (2) (Split) follows from (Recouple) and ($\Z_n$).

    \begin{equation*}\tag{$\Z_n$}
        \underbrace{ \skein{/skein_figs/n_g_strands}{0.08} }_{\text{$n$}} 
        = \skein{/skein_figs/dec_tet_3}{0.2}
    \end{equation*}

    \begin{equation*}\tag{Split}
        \underbrace{ \skein{/skein_figs/split_LHS}{0.125} }_{\text{$n$}} 
        = \frac{1}{\delta^2} \underbrace{ \skein{/skein_figs/split_RHS}{0.085} }_{\text{$n$}}
    \end{equation*}

\end{lemma}


\begin{remark}\label{rem:both-orientations}
    The previous lemma implies that, upon reversing the orientations of the lefthand sides of the relations in 
    Definition~\ref{def:zn-ext} will give similar relations.
    This fact will be used in the proof of Lemma~\ref{lem:ext-dec}.
\end{remark}

\begin{remark}
    It is worth noting the following standard abuse of language. 
    A diagrammatically presented category such as a cyclic extension has hom-spaces which are formal spans of diagrams.
    When applying a relation such as (decTrigon) locally, the result is a {\it linear combination} of diagrams.
    Usually, though, this linear combination has some desirable quality, such as a smaller number of internal faces in each summand.
    In this instance, we prefer to say something along the lines of, 
    ``applying (decTrigon) decreases the number of internal faces,''
    instead of, for instance, the more wordy,
    ``applying (decTrigon) turns this diagram into a linear combination of diagrams with fewer internal faces.''
\end{remark}

\begin{lemma}
    In a $\Z_n$-like extension, there exist $n$ scalars $s_i$ such that the following relation holds:
    \begin{equation*}\tag{Slide}
        \skein{/skein_figs/slide_LHS}{0.1} 
        = \sum_{i=0}^{n-1} s_i \skein{/skein_figs/slide_RHS2}{0.1}^i
    \end{equation*}
\end{lemma}
\begin{proof}
    Apply (decStick) to the undecorated lower leg of the trivalent vertex to conjure a down-oriented red strand.
    Then use (Recouple), and (Change of Basis).
\end{proof}

We use this fact in the proof of the following lemma.

\begin{lemma}\label{lem:ext-dec}
    In a cyclic extension, any decoration of a diagram with one internal face may be expressed as a $\C$-linear combination 
    of singly-externally decorated diagrams
\end{lemma}

\begin{proof}
    We prove the lemma for a decorated trigon, and leave the remaining cases to the reader. 
    We begin with a maximally-decorated trigon.
    All less decorated cases are absorbed along the way in this analysis.
    Additionally, following Remark~\ref{rem:both-orientations} we know that analogues of the defining relations 
    for a cyclic extension hold for both strand orientations of the diagrams on the lefthand side.
    Hence we begin with a diagram whose red strands are unoriented;
    this is possible by applying whichever version of the relations we need at the time.
    We may omit the labels $1,\dots,n-1$ for the red strands by the same reasoning:
    analogous relations hold for multiple red strands.

    Now, a maximally-decorated trigon is of the form:
    \[
        \skein{/skein_figs/dec_3gon1}{0.15}
    \]
    with any orientation on the red strands. We apply the relations (Swap) and (Slide) on the internal red strands to obtain a combination of diagrams of the form
    \[
        \skein{/skein_figs/dec_3gon2}{0.15}
    \]
    Now apply (Change of Basis) to reduce to a combination of diagrams of the form
    \[
        \skein{/skein_figs/dec_3gon3}{0.15}
    \]
    By another application of (Slide) and (Change of Basis) we arrive at a diagram of the form 
    \[
        \skein{/skein_figs/dec_3gon4}{0.15}
    \]
    During this last step, we pick up red strands between the black ``spokes''; one may happily move these out of the diagram.
\end{proof}

One more lemma will complete our ability to evaluate closed diagrams in $\Z_n$-like extensions.
\begin{lemma}\label{lem:decorated-graph}
    Suppose a diagram $\EE$ in a $\Z_n$-like extension consists only of black loops and 
    red oriented edges between them. 
    Suppose furthermore that each loop of $\EE$ has either exactly $n$ red strands entering or exactly $n$ red strands leaving.
    Then the diagram $\EE$ evaluates to a scalar.
\end{lemma}

\begin{proof}
    We'll use graph theoretic language, with black loops playing the role of nodes,
    and oriented red edges playing the role of oriented edges.

    If a node has exactly one neighbor, use ($\Z_n$) to remove both for a $\delta^2$.
    So assume every node has at least two neighbors.
    Pick one node and call it $u$.
    Without loss of generality, assume the only oriented edges leaving $u$ are outgoing.
    Call one of its neighbors $w$;
    since $u$ has multiple neighbors, we have $\deg(u\to w)=k<n$.
    Recalling briefly that $u$ is actually a loop, we may traverse it counterclockwise
    beginning at the outgoing edge to $w$.
    Along the way there will be $n-k$ more outgoing red edges;
    each corresponds to a neighbor of $u$.
    With this counterclockwise orientation, call the remaining neighbors of $u$ by $v_1,\dots,v_{n-k}$, noting that
    these need not be distinct.
    We may label the neighbors of $w$ in a similar way, but traversing clockwise
    beginning at the incoming edge from $u$.
    Call these neighbors of $w$ by $w_1,\dots,w_{n-k}$, again noting that 
    they need not be distinct.

    The diagram is planar, so we may isotope it to look, locally, as follows:
    \begin{center}
        \begin{tikzpicture}
            \node[shape=circle,draw=black] (u) at (6,0) {$u$};

            \node[shape=circle,draw=black] (v1) at (5,1.5) {$v_1$};
            \node[shape=circle,draw=black] (v2) at (4,1.5) {$v_2$};
            \node[shape=circle,draw=black] (vnk) at (2,1.5) {$v_{n-k}$};

            \node[shape=circle,draw=black] (w) at (7,1.5) {$w$};
            \node[shape=circle,draw=black] (w1) at (7,3) {$w_1$};
            \node[shape=circle,draw=black] (w2) at (8,4) {$w_2$};
            \node[shape=circle,draw=black] (wnk) at (10,6) {$w_{n-k}$};

            \node[] (vdots) at (3,1.5) {$\dots$};
            \node[rotate=45] (wdots) at (9,5) {$\dots$};

            \path [->, draw=orange] (u) edge node {} (v1);
            \path [->, draw=orange] (u) edge node {} (v2);
            \path [->, draw=orange] (u) edge node {} (vnk);

            \path [->, draw=orange] (u) edge node [right] {$k$} (w);
 
            \path [->, draw=orange] (w1) edge node {} (w);
            \path [->, draw=orange] (w2) edge [bend left] node {} (w);
            \path [->, draw=orange] (wnk) edge [bend left] node {} (w);
        \end{tikzpicture}
    \end{center}
    Note that we may have omitted edges here.
    That is, we may have $\deg(v_i \to w_j) \neq 0$, or there may be other nodes not pictured.
    This is not an issue for us.

    Now apply (Recouple), exchanging pairs of edges $u\to v_i$ and $w_j \to w$ for pairs of edges $u\to w$ and $w_j\to v_i$.
    This changes $\deg(u\to w)$ to $n$, allowing us, using ($\Z_n$), to exchange a pair of nodes for a scalar.
    Continue until only pairs of nodes remain, exchanging each pair for a $\delta^2$.
\end{proof}

\begin{theorem}\label{thm:eval-criteria}
    A $\Z_n$-like extension of an Euler-evaluable trivalent category is evaluable.
\end{theorem}

\begin{proof}
    Suppose we begin with a diagram given by a closed, decorated planar trivalent graph.
    Begin by applying relations from the underlying trivalent category's evaluation algorithm to any undecorated faces; 
    this decreases the number of trivalent vertices.
    By the standard Euler characteristic calculation, there must remain some black $m$-gon with $m\in\{ 0,1,\dots,5 \}$.
    The $m=0$ case is handled at the end of this proof.
    The $m=1$ case is handled by (Change of Basis) and (Schur 0). 
    Choose one such face and apply Lemma~\ref{lem:ext-dec} to reduce it to a singly-externally-decorated $m$-gon.
    Now one of the relations (decBigon), (decTrigon), (decTetragon), or (decPentagon) allows us to pop the face.
    This process decreases the number of faces (ignoring red strands) in diagrams by at least 1 at every step,
    but also may increase the number of connected components in any summand.
    Continue this process until only decorated loops, or decorated loops connected by red strands remain.
    If only decorated loops remain, apply (decStick).
    
    Our diagram now consists of a number of black loops, connected by red strands.
    Use (Recouple) and ($\Z_n$) to make it so every black loop has either only in-strands or only out-strands attached to it.
    Now, if any black loop has more or fewer than $n$ strands entering or exiting it,
    then (Split), ($\Z_n$), and (Schur 0) imply the whole diagram is zero.
    So suppose each black loop has exactly $n$ strands entering or exiting.
    Apply Lemma~\ref{lem:decorated-graph} to evaluate the remaining graph for a scalar.
\end{proof}

As with all evaluability arguments, if we have nontriviality, we immediately deduce simplicity of the unit.

\begin{corollary}\label{cor:Zn-simple-unit}
    For a $\Z_n$-like extension $\DD$ of an Euler-evaluable trivalent category, we have
    \[
        \dim\Hom_{\DD} (0\to 0) \leq 1.
    \]
\end{corollary}

\section{GPA Embeddings}\label{sec:gpa-emb}

This section is devoted to discussing the details of our GPA embeddings.
This will include a discussion of the techniques used to solve the defining equations,
along with a discussion of the coordinates these solutions define.
Subsection~\ref{subsec:extend-level-4} discusses some of the representation theory which led us to 
search for $\Z_n$-like extensions in the first place;
in this respect, we discuss only the details of level 4, but the story at level 3 is essentially the same.

We find our module fusion graphs by orbifolding the graphs in Figures 18b and 21b of \cite{g2_graphs}.
These graphs are shown in Figure~\ref{fig:new-graphs}.

\begin{figure}
    \noindent\makebox[\textwidth]{
    \begin{tikzpicture}[scale=1.0]
        \node[shape=circle,draw=black] (A) at (0,1.5) {1};
        \node[shape=circle,draw=black] (B) at (1.25,3) {2};
        \node[shape=circle,draw=black] (C) at (4.25,3) {3};
        \node[shape=circle,draw=black] (D) at (5.5,1.5) {4};
        \node[shape=circle,draw=black] (E) at (4.25,0) {5};
        \node[shape=circle,draw=black] (F) at (1.25,0) {6};

        \path (B) edge [loop, in=45, out=135, looseness=8] node {} (B);
        \path (C) edge [loop, in=45, out=135, looseness=8] node {} (C);
        \path (E) edge [loop, in=225, out=315, looseness=8] node {} (E);
        \path (F) edge [loop, in=225, out=315, looseness=8] node {} (F);

        \path [-] (A) edge node {} (B);
        \path [-] (A) edge node {} (F);

        \path [-] (B) edge node {} (C);
        \path [-] (B) edge node {} (E);
        \path [-] (B) edge node {} (F);

        \path [-] (C) edge node {} (D);
        \path [-] (C) edge node {} (E);
        \path [-] (C) edge node {} (F);

        \path [-] (D) edge node {} (E);

        \path [-] (E) edge node {} (F);

        \path [->, draw=orange] (A) edge [loop, in=135, out=225, looseness=8] node {} (A);
        \path [->, draw=orange] (D) edge [loop, in=45, out=-45, looseness=8] node {} (D);

        \path [->, draw=orange] (B) edge [bend right] node  {} (F);
        \path [->, draw=orange] (F) edge [bend right] node {} (B);

        \path [->, draw=orange] (C) edge [bend right] node {} (E);
        \path [->, draw=orange] (E) edge [bend right] node {} (C);
    \end{tikzpicture}
    \begin{tikzpicture}[scale=1.0]
        \node[shape=circle,draw=black] (A) at (1.5,1.66) {1};
        \node[shape=circle,draw=black] (B) at (0,3) {2};
        \node[shape=circle,draw=black] (C) at (3,3) {3};
        \node[shape=circle,draw=black] (D) at (1.5,0) {4};
        
        \path (B) edge [loop, in=120, out=180, looseness=10] node {} (B);
        \path (C) edge [loop, in=0, out=60, looseness=10] node {} (C);
        \path (D) edge [loop, in=240, out=300, looseness=10] node {} (D);
        
        \path [-] (D) edge node {} (B);
        \path [-] (B) edge node {} (C);
        \path [-] (C) edge node {} (D);

        \path [-] (A) edge node {} (B);
        \path [-] (A) edge node {} (C);
        \path [-] (A) edge node {} (D);

        \path [->, draw=orange] (B) edge [bend left] node {} (C);
        \path [->, draw=orange] (C) edge [bend left] node {} (D);
        \path [->, draw=orange] (D) edge [bend left] node {} (B);
        \path [->, draw=orange] (A) edge [loop] node {} (A);
    \end{tikzpicture}
    }
    \caption{
        Fusion graphs at level 4 and 3 for $Y_4$ and $Y_3$ (black) and $g_4$ and $g_3$ (orange), respectively. 
        See Figures 21b and 18b, respectively, of \cite{g2_graphs}.
        }
    \label{fig:new-graphs}
\end{figure}
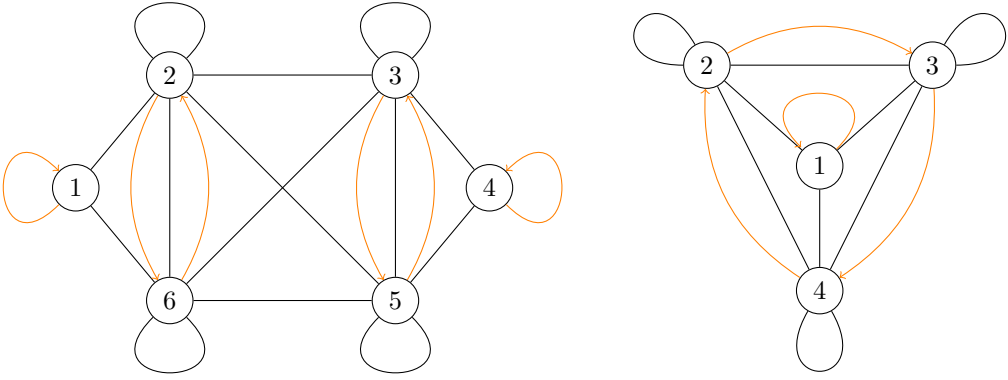

        
        



One may give a monoidal functor $F:\GG_2(q) \to \GPA(\Gamma)$ by specifying the image of the morphism 
\[
F\left( \skein{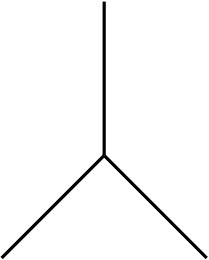}{0.08} \right) \in \Hom_{\GPA(\Gamma)}(2\to1).
\]
This amounts to giving a list of $M\coloneqq\tr(\Gamma^2\cdot\Gamma)$ complex numbers\footnote{
    We freely switch between using $\Gamma$ to mean the graph itself and the graph's adjacency matrix. }, 
say $a_1,\dots,a_M$. 
Pushing the defining relations of $\GG_2(q)$ through $F$, we see that 
these complex numbers satisfy equations in the $a_i$ and $\ol{a_i}$. 
If we assume for now that each $a_i$ is real, then this reduces the system to a collection 
of polynomials in the $a_i$. \footnote{
    This assumption is useful only if it turns out to help us solve the system. 
    In fact, any assumptions we make about this system, if they yield solutions, are valid.
    }
Once we have the image of the trivalent vertex in hand, we have found an embedding of the planar algebra it generates. 
We can then follow a similar approach to solve for the image
\[
F\left( \skein{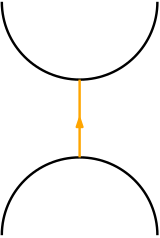}{0.08} \right) \in \Hom_{\GPA(\Gamma)}(2\to2)
\]
to extend the GPA embedding of $\GG_2(q)$.
Let us discuss our examples.

\subsection{Level 4 Trivalent Embedding}
We will begin with level 4.
Let $\Gamma_4$ be the graph on the left side of  Figure~\ref{fig:new-graphs}.
Set $q_4 \coloneq e^{2\pi i/48}$.
The following theorem says that we have an embedding of $\ol{\GG_2(q_4)}$ into the GPA on $\Gamma_4$.
A proof of the theorem is extremely straightforward; one must verify some equations.
Why this verification amounts to a proof requires some explanation, and is discussed in detail 
below the theorem.

\begin{theorem}\label{thm:G2-level4-GPA-emb}
    There is a faithful monoidal functor $\ol{F_4}:\ol{\GG_2(q_4)} \hookrightarrow \GPA(\Gamma_4)$.
\end{theorem}
\begin{proof}
    We first construct a monoidal functor $F_4: \GG_2(q_4) \hookrightarrow \GPA(\Gamma_4)$, and then 
    apply Lemma~\ref{lem:simple-unit-unitary}.

    Let 
    \[
        (p_1,q_1),\dots,(p_M,q_M)
    \]
    be the defining basis\footnote{ 
        See the attached Mathematica files for the specific ordering chosen. } 
    for $\Hom_{\GPA(\Gamma_4)}(2\to1)$; at level 4 we have $M=88$. 
    Then it must be that 
    \[
        F_4 \left( \skein{skein_figs/trivalent}{0.08} \right) = a_1(p_1,q_1)+\cdots+ a_M(p_M,q_M)
    \]
    for some $a_1,\dots,a_M \in \C$.
    The \ref{eq:Bigon} relation, when sent through $F$, becomes the system
    \[
        \sum_{i=1}^M a_i(p_i,q_i) \circ \sum_{j=1}^M a_j(q_j,p_j) = k^2 \sum_{e\in E(\Gamma)} (e,e).
    \]
    This system is quadratic in the $a_i$ since it involves up to two trivalent vertices on either side. 
    Similarly, the (Schur 0, 1) and (Rotate) relations therefore determine a system of linear equations; 
    the (Trigon), (Tetragon), and (Pentagon) give cubic, quartic, and quintic equations, respectively. 
    To show $F_4$ is a monoidal functor amounts to showing that the equations in the $a_i$ induced 
    by the relations in $G_2(q_4)$ are all satisfied.
    This is verified in the Mathematica notebook \verb|/code/level-4/nb-6_verify-g.nb|.
\end{proof}

The above proof does not mention the process used to arrive at such a definition of $F_4$.
We briefly mention it here.
For more details, see the author's PhD thesis \cite{hill_thesis}.

To go about solving these equations, it is often useful to solve the linear subsystem first and substitute the solution into the quadratic equations. 
For example, when we substitute the linear solution into the (Bigon) and (Tetragon) equations, 
we are able to isolate the following resulting equations:

\begin{equation*}
    a_{8}^2+a_{85}^2 = 4-\sqrt{2}+2 \sqrt{3}-\sqrt{6} 
\end{equation*}

\begin{equation*}
    a_{69}^2+\left(1+\sqrt{\frac{3}{2}}\right) a_{8}^2 = \frac{3+\sqrt{3}+\sqrt{6}}{\sqrt{2}} 
\end{equation*}

\begin{equation*}
    a_{69}^2 \left(\left(2+\sqrt{6}\right) a_{8}^2+\left(2+\sqrt{6}\right)
   a_{85}^2-2 \sqrt{2+\sqrt{3}}\right) = 5+\sqrt{2}+\sqrt{3}+2 \sqrt{6}
\end{equation*}

\begin{equation*}
    2 a_{69}^4+\left(5+2 \sqrt{6}\right)
   a_{85}^4 = \left(3+\sqrt{2}+\sqrt{3}+\sqrt{6}\right) a_{85}^2+3
   \sqrt{6}+\sqrt{3}+2 \sqrt{2}+7 
\end{equation*}

Up to three choices of sign, the solution to this system is
\begin{align*}
    a_{8} & = \sqrt{2+\sqrt{3}-\sqrt{2+\sqrt{3}}} \\
    a_{69} & = \sqrt{\frac{1}{2} \left(-1+\sqrt{2}+\sqrt{3}\right)} \\
    a_{85} & = \sqrt{2+\sqrt{3}-\sqrt{2+\sqrt{3}}} \\
\end{align*}
Similar equations containing $a_{31}$, $a_{55}$, and $a_{63}$ appear as well. 
We repeat the process and obtain the additional values
\begin{align*}
    a_{31} & = \sqrt{2+\sqrt{3}-\sqrt{2+\sqrt{3}}} \\
    a_{55} & = \sqrt{1-\sqrt{\frac{3}{2}}+\frac{1}{\sqrt{2}}} \\
    a_{63} & = \sqrt{2+\sqrt{3}-\sqrt{2+\sqrt{3}}} \\
\end{align*}
These six values begin a cascade of equation solving. 
They, along with the linear solution, reduce many of the original high-order equations to linear. 
We solve those, then repeat the process until we're forced to confront nonlinearity. 
The nonlinearity we encounter forces us to extract square roots, and ending up with a few degree-16 algebraic numbers. 
For instance,
\[
    a_{10} = \frac{1}{2} \left(\sqrt{1+\sqrt{6-3 \sqrt{3}}}+\sqrt{\sqrt{2+\sqrt{3}}-1}\right).
\]
Up to sign, the coordinates of $F_4 \left( \skein{skein_figs/trivalent}{0.08} \right)$ 
take on the following values:
\begin{align*}
    \alpha_1 & = \sqrt{\frac{1}{2} \left(1+2 \sqrt{2}+\sqrt{3}+\sqrt{6}\right)} \\
    \alpha_2 & = \sqrt{\frac{1}{2} \left(-1+\sqrt{2}+\sqrt{3}\right)} \\
    \alpha_3 & = \sqrt{\frac{3}{2} \left(\sqrt{2+\sqrt{3}}-1\right)} \\
    \alpha_4 & = \sqrt{2+\sqrt{3}-\sqrt{2+\sqrt{3}}} \\
    \alpha_5 & = \sqrt{\frac{1}{2} \left(\sqrt{2+\sqrt{3}}-1\right)}  \\
    \alpha_6 & = \sqrt{\frac{1}{2} \left(\sqrt{3}+\sqrt{2+\sqrt{3}}\right)} \\
    \alpha_7 & = \sqrt{1-\sqrt{\frac{3}{2}}+\frac{1}{\sqrt{2}}}  \\
    \alpha_8 & = \frac{1}{2} \left(\sqrt{1+\sqrt{6-3 \sqrt{3}}}-\sqrt{\sqrt{2+\sqrt{3}}-1}\right) \\
    \alpha_9 & = \frac{1}{2} \left(\sqrt{1+\sqrt{6-3 \sqrt{3}}}+\sqrt{\sqrt{2+\sqrt{3}}-1}\right). \\
\end{align*}
These coordinates of $F_4 \left( \skein{skein_figs/trivalent}{0.08} \right)$ give us a definition of $F_4$.

\subsection{Extension of level 4}\label{subsec:extend-level-4}
With our embedding of $\GG_2(q_4)$, i.e. the coordinates of $F_4 \left( \skein{skein_figs/trivalent}{0.08} \right)$, in hand, 
we now know where to find a subcategory of $\GPA(\Gamma_4)$ isomorphic to $\ol{\GG_2(q_4)}$.
In practice, we are relying on the fact that the free functor gives an embedding 
\[
    \ol{\Rep(U_{q_4}(\gg_2))} \hookrightarrow \ol{\Rep(U_{q_4}(\gg_2))}_{A_4}
\]
which takes the simple $\otimes$-generator $X_4$ to a simple $\otimes$-generator $Y_4\coloneq \FF_{A_4}(X_4)$.\footnote{
    Simplicity of $Y_4$ may be easily checked using the $X_4$ fusion graph of $\ol{\Rep(U_{q_4}(\gg_2))}$. }
Combine this with the facts that $\ol{\Rep(U_{q_4}(\gg_2))}_{A_4}$ contains a subcategory equivalent to $\Vecc{\Z_2}$,
and that both of the grouplike simples are subobjects of the square of the $\otimes$-generator $Y_4$.
One concludes that $\Hom_{\ol{\Rep(U_{q_4}(\gg_2))}_{A_4}}(Y_4^{\otimes2}\to Y_4^{\otimes2})$ 
contains a projection onto a $\Z_2$-like simple object 
which follows relations analogous to (Recouple), (Schur 0, 1), (Swap), and (decStick).
Now, to enlarge our copy of $\ol{\GG_2(q_4)}$ to a $\Z_2$-like extension, we begin by finding 
an element of $\Hom_{\GPA(\Gamma_4)}(2\to 2)$ which satisfies the relations
(Recouple), (Schur 0, 1), (Swap), and (decStick); 
the role will be played by $\skein{skein_figs/dec_tet_5}{0.08}$.

To this end, we take an approach similar to the one of the previous subsection.
We may identify $\GG_2(q_4)$ with the trivalent subcategory of $\DD_4$,
and now view $F_4$ as a functor out of this subcategory.
To finish the definition of $F_4$, we must find the image
\[
    F_4\left( \skein{skein_figs/dec_tet_5}{0.08} \right) \in \Hom_{\GPA(\Gamma_4)}(2\to2)
\]
satisfying (Recouple), (Schur 0, 1), (Swap), and (decStick).
We follow a process similar to solving for the image $F_4\left( \skein{skein_figs/trivalent}{0.08} \right)$.
That is, we assume that for some $b_1,\dots,b_N \in \C$ where $N \coloneq \tr(\Gamma_4^2 \cdot \Gamma_4^2) =400$ 
at level 4, we have
\[
    F_4\left( \skein{skein_figs/dec_tet_5}{0.08} \right) = \sum_{i=1}^N b_i (p'_i,q'_i).
\]
We may push the relations (Schur 0), (Schur 1), and (decStick) through $F_4$, and obtain 
a set of linear equations in the $b_i$.
We additionally obtain a large linear system which follows from Lemma~\ref{lem:general-half-braid}.
In the present context, this lemma states that our projection $\skein{skein_figs/dec_tet_5}{0.08}$ 
satisfies the relation
\begin{equation*}\tag{Half-braid}
    \skein{/skein_figs/hb_top}{0.11} = \skein{/skein_figs/hb_bottom}{0.15}
\end{equation*}
where the unoriented red strand indicates that the equation 
holds whenever we pick the same orientation for the left and right sides.
Solving these new equations gives us the coordinates of the projection $F_4\left( \skein{skein_figs/dec_tet_5}{0.08} \right)$.

\begin{theorem}\label{thm:init-D4-GPA-emb}
    There exists an element $P_4 \in\Hom_{\GPA(\Gamma_4)}(2\to 2)$ satisfying the relations
    (Schur 0, 1), (decStick), and (Half-braid)
\end{theorem}
\begin{proof}
    See the mathematica notebook \verb|/code/level-4/nb-6_verify-g.nb| for verification.
\end{proof}

Tables~\ref{tab:lvl-4-proj-coefs-1} and \ref{tab:lvl-4-proj-coefs-2} hold numerical approximations 
of the nonzero projection coordinates.
There are blocks of nonzero coordinates of length 4 and 25. 
These sizes, and the location of the nonzero real coordinates follow naturally 
when one considers Remark~\ref{rem:Pg-properties}.
Recall that the defining bases for the spaces 
\[
    \Hom_{\GPA(\Gamma)}(m\to n)
\]
are given in terms of pairs of paths. 
The (undirected) graphs we are using have at most a single edge between any two vertices. 
Hence an edge is equivalent to a pair of vertices, and a path is equivalent to an ordered tuple of vertices. 
For example, the path
\[
    p = v_1 \longrightarrow v_2 \longrightarrow v_3
\]
is equivalent to the ordered triple $(v_1,v_2,v_3)$. 

The only coordinates of the projection in the GPA which are nonzero are at those basis vectors 
\[
    (i \to\blank\to j, i \to\blank\to j)
\]
where $i\to j$ is a directed edge of the $g$-fusion graph. 
For $i=j=1,4$ there are two possible values for $\blank$; pairing them gives 4 pairs. 
For $i,j\neq 1,4$ there are five possible values for $\blank$; pairing them gives 25 pairs. 
The columns of Tables~\ref{tab:lvl-4-proj-coefs-1} and \ref{tab:lvl-4-proj-coefs-2} give the 
values of the coordinates of the projection, 
with dictionary ordering on the pairs of $\blank$ values. 
That is, the column labeled by $1\to\blank\to1$ shows the coordinates on the ordered basis
\begin{align*}
    (1\to 2 \to1, 1\to 2 \to1) \\
    (1\to 2 \to1, 1\to 3 \to1) \\
    (1\to 3 \to1, 1\to 2 \to1) \\
    (1\to 3 \to1, 1\to 3 \to1) \\
\end{align*}

With this ordering in mind, and recalling that the GPA's dagger operation swaps paths, 
the conjugate pairs appear where one would expect them.

\begin{table}
    \centering
    \begin{tabular}{|cc|} \hline
        $1\to\blank\to1$ & $4\to\blank\to4$ \\ \hline\hline
        $2.22$  & $2.22$ \\
        $-2.22$ & $-2.22$ \\
        $-2.22$ & $-2.22$ \\
        $2.22$  & $2.22$ \\ \hline
    \end{tabular}
    \caption{The size 4 blocks of nonzero projection coordinates.}
    \label{tab:lvl-4-proj-coefs-1}
\end{table}

\begin{table}
    \noindent\makebox[\textwidth]{
    \begin{tabular}{|cccc|} \hline
        $2\to\blank\to6$                 & $3\to\blank\to5$         & $5\to\blank\to3$           & $6\to\blank\to2$ \\ \hline\hline
        $0.44949$                        & $1$                      & $1$                        & $0.44949$ \\
        $0.474073 - 0.474073 i$          & $0.825482 + 0.564429 i$  & $0.825482 - 0.564429 i$    & $0.474073 + 0.474073 i$ \\
        $-0.123758 - 0.658919 i$         & $0.123758 + 0.658919 i$  & $0.123758 - 0.658919 i$    & $-0.123758 + 0.658919 i$ \\
        $-0.123758 + 0.658919 i$         & -$0.564429 + 0.825482 i$ & -$0.564429 - 0.825482 i$   & $-0.123758 - 0.658919 i$ \\
        $0.474073 + 0.474073 i$          & $0.931852 - 0.362839 i$  & $0.931852 + 0.362839 i$    & $0.474073 - 0.474073 i$ \\
        $0.474073 + 0.474073 i$          & $0.825482 - 0.564429 i$  & $0.825482 + 0.564429 i$    & $0.474073 - 0.474073 i$ \\
        $1$                              & $1$                      & $1$                        & 1 \\
        $0.564429 - 0.825482 i$          & $0.474073 + 0.474073 i$  & $0.474073 - 0.474073 i$    & $0.564429 + 0.825482 i$ \\
        $-0.825482 + 0.564429 i$         & $i$                      & $-i$                       & $-0.825482 - 0.564429 i$ \\ 
        $i$                              & $0.564429 - 0.825482 i$  & $0.564429 + 0.825482 i$    & -i \\
        $-0.123758 + 0.658919 i$         & $0.123758 - 0.658919 i$  & $0.123758 + 0.658919 i$    & $-0.123758 - 0.658919 i$ \\
        $0.564429 + 0.825482 i$          & $0.474073 - 0.474073 i$  & $0.474073 + 0.474073 i$    & $0.564429 - 0.825482 i$ \\
        $1$                              & $0.44949$                & $0.44949$                  & 1 \\
        $-0.931852 - 0.362839 i$         & $0.474073 + 0.474073 i$  & $0.474073 - 0.474073 i$    & $-0.931852 + 0.362839 i$ \\
        $-0.825482 + 0.564429 i$         & $-0.123758 - 0.658919 i$ & $-0.123758 + 0.658919 i$   & $-0.825482 - 0.564429 i$ \\
        $-0.123758 - 0.658919 i$         & -$0.564429 - 0.825482 i$ & -$0.564429 + 0.825482 i$   & $-0.123758 + 0.658919 i$ \\ 
        $-0.825482 - 0.564429 i$         & $-i$                     & $i$                        & $-0.825482 + 0.564429 i$ \\                 
        $-0.931852 + 0.362839 i$         & $0.474073 - 0.474073 i$  & $0.474073 + 0.474073 i$    & $-0.931852 - 0.362839 i$ \\                 
        $1$                              & $1$                      & $1$                        & $1$ \\                 
        $0.564429 - 0.825482 i$          & $-0.825482 - 0.564429 i$ & $-0.825482 + 0.564429 i$   & $0.564429 + 0.825482 i$ \\                 
        $0.474073 - 0.474073 i$          & $0.931852 + 0.362839 i$  & $0.931852 - 0.362839 i$    & $0.474073 + 0.474073 i$ \\                 
        $-i$                             & $0.564429 + 0.825482 i$  & $0.564429 - 0.825482 i$    & $i$ \\                 
        $-0.825482 - 0.564429 i$         & $-0.123758 + 0.658919 i$ & $-0.123758 - 0.658919 i$   & $-0.825482 + 0.564429 i$ \\                 
        $0.564429 + 0.825482 i$          & $-0.825482 + 0.564429 i$ & $-0.825482 - 0.564429 i$   & $0.564429 - 0.825482 i$ \\                 
        $1$                              & $1$                      & $1$                        & $1$ \\ \hline                 
    \end{tabular}
    }
    \caption{The size 25 blocks of nonzero projection coordinates.}
    \label{tab:lvl-4-proj-coefs-2}
\end{table}

\subsection{New Relations at Level 4}
We open this subsection with its primary result, then introduce the techniques we used to arrive at it.

\begin{theorem}\label{thm:D4-new-relns}
    The element $P_4 \in\Hom_{\GPA(\Gamma_4)}(2\to 2)$ satisfies the defining relations of $\DD_4$
    given in Definition~\ref{def:D4}.
\end{theorem}

Our initial discovery, Theorem~\ref{thm:init-D4-GPA-emb}, makes no mention of the relations 
(Swap), (Change of Basis), (decTrigon), (decTetragon), or (decPentagon).
That is, the structure constants in these relations were unknown before exploring our GPA embedding.

In order to discover these new relations,
we hypothesize the form such a relation should take, impose that form, and solve for the structure constants.
For example we would suppose, based the simplicity of $Y_4$ and the invertibility of $g_4$, that an equation of the form
\[
    F_k\left( \skein{/skein_figs/swap_LHS}{0.15} \right) = \omega F_k\left( \skein{/skein_figs/swap_RHS}{0.15} \right)
\]
holds, for some $\omega$.
By ($\Z_n$), $\omega$ must be an $n$-th root of 1.
With our explicit GPA embeddings of the projection in hand, discovering what $\omega$ is becomes a matter of solving
a system of linear equations for one variable.

Discovering the (decTrigon) structure constants is similarly reduced to solving a system of linear equations of the form
\[
F_k\left( \skein{/skein_figs/dec_trigon_LHS}{0.1} \right)= 
        t_1 F_k\left( \skein{/skein_figs/trivalent}{0.1} \right) 
        + t_2 F_k\left( \skein{/skein_figs/triv_rightUp}{0.1} \right)
\]
for $t_1$ and $t_2$.
The same principle applies when searching for the relations (decTetragon) and (decPentagon).
In practice, discovering that, e.g., the coefficient of the summand \skein{/skein_figs/dec_pent_RHS1}{0.07}
in the (decPentagon) relation is $q^8 - q^4 + \frac{1}{q^4} + \frac{1}{q^2}$ 
required the use of some novel computational algebraic number theory.
For more details, see the author's PhD thesis \cite{hill_thesis}.

As an immediate consequence of Theorems~\ref{thm:G2-level4-GPA-emb} 
and \ref{thm:D4-new-relns} we have the following corollary.

\begin{corollary}\label{cor:D4-unitary}
    The category $\DD_4$ is a nontrivial $\Z_2$-like extension of $\GG_2(q_4)$, 
    and the semisimple quotient $\ol{\DD_4}$ is unitary.
\end{corollary}
\begin{proof}
    We deduce $\DD_4$ is nonzero by its embedding into a nonzero subcategory of $\GPA(\Gamma_4)$.
    Unitarity of its semisimple quotient follows from noting the unitarity of $\GPA(\Gamma_4)$,
    and applying Lemma~\ref{lem:simple-unit-unitary}.
\end{proof}

\begin{corollary}\label{cor:D4-dagger-emb}
    The embedding $\GG_2(q_4) \hookrightarrow \DD_4$ descends to a 
    $\dagger$-embedding $\ol{\GG_2(q_4)} \hookrightarrow \ol{\DD_4}$.
\end{corollary}
\begin{proof}
    Use Lemma~\ref{lem:simple-unit-unitary} again, and compose the induced functor with the functor
    \[
        \DD_4 \onto \ol{\DD_4}
    \]
    onto the semisimple quotient.
\end{proof}
Later, we will take the Karoubi completion of this embedding.
For now, we turn to a discussion of how we obtain the $\DD_4$ structure constants.

\subsection{Level 3 Trivalent Embedding}
We now tell a similar story, but at level 3.
However, for the trivalent coefficients we use numerical approximations here, 
and relegate the actual numbers to the attached Mathematica files.
We were unable to find reader-friendly representations of the GPA-embedding coordinates or structure constants.
The coordinates for the trivalent GPA embedding were algebraic numbers of degree 12 or 24.
Even worse, the structure constants for the relations (Change of Basis), (decTrigon), (decTetragon), and (decPentagon) 
are all of the form
\[
    c_1 + c_2 \alpha^2 + c_3 q_3 + c_4 \alpha^2 q_3
\]
where $c_i\in \Q(q_3+q_3^{-1})$, the maximal real subfield of $\Q(q_3)$, 
and $\alpha$ is an algebraic number of degree 24 which appears as a coordinate of 
$F_3 \left( \skein{skein_figs/trivalent}{0.08} \right)$.
For the relations (Change of Basis) and (decTrigon), the power basis coordinates of the $c_i$ are lowest form rational numbers 
whose numerators have one or two digits.
For (decTetragon), the numerators and denominators of the power basis coordinates of the $c_i$ have around 10 digits on average.
In the (decPentagon) relation, this digit count explodes to around 135.

Let $\Gamma_3$ be the graph given in the right side of Figure~\ref{fig:new-graphs}.
Set $q_3 \coloneq e^{2\pi i/42}$.
The following result gives us a GPA embedding of $\ol{\GG_2(q_3)}$.

\begin{theorem}\label{thm:G2-level3-GPA-emb}
    There is a faithful monoidal functor $\ol{F_3}: \ol{\GG_2(q_3)} \to \GPA(\Gamma_3)$.
\end{theorem}
\begin{proof}
    Similar to Theorem~\ref{thm:G2-level4-GPA-emb}.
    See the Mathematica notebook \[\verb|/code/level-4/nb-6_verify-g.nb|\] for 
    verification of the necessary equations.
\end{proof}

Despite the more difficult numbers, we are able to recover some of the structure of the fusion graph in the GPA coordinates.
Given a path $p$, which paths $q$ pair validly with $p$ to form a basis element of the $2\to 1$ hom-space of a GPA? 
Well, by definition, $q$ must be parallel to $p$; i.e. the sources and targets of $p$ and $q$ must coincide. 
It follows that the only valid pairing for such $p$ is
\[
    q = v_1 \longrightarrow v_3,
\]
which may also be represented as $(v_1,v_3)$. 
So the only $2\to 1$ basis element which $p$ appears in is
\[
    \big( (v_1,v_2,v_3), (v_1,v_3) \big).
\]
But the parallel condition defining basis elements makes including $(v_1,v_3)$ redundant; 
we might just as well have called the basis element by 
\[
    (v_1,v_2,v_3).
\]
This is how we refer to $2\to 1$ GPA basis elements. 
Indeed, in Table~\ref{tab:lvl-3-triv-coefs}, the first two columns combine to specify which basis elements are being specified, 
and the third column gives the approximate coordinate of the trivalent embedding on that basis element. 
For example, the first row of Table~\ref{tab:lvl-3-triv-coefs} tells us that the 
coordinate of the $(2,2,2)$ basis element is approximately $1.08393$; 
the second row tells us that the coordinate of the $(4,2,3)$ basis element is approximately $0.619371$. 

Paths of the form $(i,j,i)$, $(i,i,j)$, or $(i,j,j)$ for $i,j\neq1$ require a bit more care to describe. 
There is nontrivial interplay with the graph symmetry swapping vertices 2 and 4. 
When these two vertices are swapped, a path whose coordinate has absolute value $0.155691$ is sent to 
one whose coordinate has absolute value $1.69414$. 
The nine paths whose coordinates have absolute value $0.155691$ are:
\[
    (2,3,3),\hspace{0.2em} (3,3,2),\hspace{0.2em} (3,2,3),\hspace{0.2em} (2,4,2),\hspace{0.2em} (4,3,4),\hspace{0.2em} (2,2,4),\hspace{0.2em} (3,4,4),\hspace{0.2em} (4,2,2),\hspace{0.2em} (4,4,3) 
\]
One may use the symmetry $2 \xleftrightarrow{} 4$ to find the rest of the coordinates.

\begin{table}
    \centering
    \begin{tabular}{|cc|c|} \hline
        Vertex Path & Conditions & Coefficient \\ \hline\hline
        $(i,i,i)$ &  $i\neq1$ & 1.08393 \\[10pt]  \hline
        $(i,j,k)$ &  $\{i,j,k\}=\{2,3,4\}$   & 0.619371 \\[10pt] \hline
        $(i,1,k)$ &  $i,k\neq 1$, $i\neq k$ & 1.69414 \\[10pt] \hline
        $(i,1,i)$ &  $i\neq 1$   & 0.861006 \\[10pt] \hline
        $(i,i,1)$ or $(1,i,i)$ &  $i\neq1$   & 0.967919 \\[10pt] \hline
    \end{tabular}
    \caption{Level 3 trivalent embedding coefficients.}
    \label{tab:lvl-3-triv-coefs}
\end{table}

\subsection{Extension of level 3}
We quickly give the analogous results to Theorem~\ref{thm:D4-new-relns} 
and Corollaries~\ref{cor:D4-unitary} and \ref{cor:D4-dagger-emb}.
However, we skip the intermediate version of Theorem~\ref{thm:init-D4-GPA-emb}.

\begin{theorem}\label{thm:D3-GPA-emb}
    There exists an element $P_3 \in\Hom_{\GPA(\Gamma_3)}(2\to 2)$ satisfying the $\Z_3$-like extension relations, 
    with structure constants given in the Mathematica notebook \verb|/code/level-3/nb-6_verify-g.nb|.
\end{theorem}
\begin{proof}
    This result is again proved by direct verification of the required equations.
    See the Mathematica notebook \verb|/code/level-3/nb-6_verify-g.nb|.
\end{proof}

Similarly to the level 4 case, Theorems~\ref{thm:G2-level3-GPA-emb} and \ref{thm:D3-GPA-emb}
imply nontriviality and unitarity of $\ol{\DD_3}$.

\begin{corollary}\label{cor:D3-unitary}
    The category $\DD_3$ is a nontrivial $\Z_3$-like extension of $\GG_2(q_3)$, 
    and the semisimple quotient $\ol{\DD_3}$ is unitary.
\end{corollary}

\begin{corollary}\label{cor:D3-dagger-emb}
    The embedding $\GG_2(q_3) \hookrightarrow \DD_3$ descends to a 
    $\dagger$-embedding $\ol{\GG_2(q_3)} \hookrightarrow \ol{\DD_3}$.
\end{corollary}

\begin{table}
    \noindent\makebox[\textwidth]{
    \begin{tabular}{|cccc|} \hline
        $1\to\blank\to1$                    & $2\to\blank\to4$                  & $3\to\blank\to2$               & $4\to\blank\to3$ \\ \hline\hline
        $1.26376$              & $0.791288$                & $0.791288 $              & $0.791288$ \\
        $-0.631881-1.09445 i$   & $0.567622 - 0.684904 i$   & $0.876955 + 0.149123 i$  & $0.674406 + 0.580055 i$ \\
        $-0.631881+1.09445 i$   & $0.674406 + 0.580055 i$   & $0.567622 - 0.684904 i$  & $0.876955 + 0.149123 i$ \\
        $-0.631881+1.09445 i$   & $-0.876955 - 0.149123 i$  & $-0.674406 - 0.580055 i$ & $0.567622 - 0.684904 i$ \\
        $1.26376$               & $0.567622 + 0.684904 i$   & $0.876955 - 0.149123 i$ & $0.674406 - 0.580055 i$ \\
        $-0.631881-1.09445 i$   & $1$                       & $1$                      & $1$ \\
        $-0.631881-1.09445 i$   & $-0.0182917 + 0.999833 i$ & $0.5 - 0.866025 i$       & $0.856735 - 0.515757 i$ \\
        $-0.631881+1.09445i$    & $-0.5 - 0.866025 i$       & $-0.856735 - 0.515757 i$ & $-0.0182917 - 0.999833 i$ \\
        $1.26376$               & $0.674406 - 0.580055 i$   & $0.567622 + 0.684904 i$  & $0.876955 - 0.149123 i$ \\ 
                              & $-0.0182917 - 0.999833 i$ & $0.5 + 0.866025 i$       & $0.856735 + 0.515757 i$ \\
                              & $1$                       & $1$                      & $1$ \\
                              & $-0.856735 + 0.515757 i$  & $0.0182917 - 0.999833 i$ & $0.5 - 0.866025 i$ \\
                              & $-0.876955 + 0.149123 i$  & $-0.674406 + 0.580055 i$ & $0.567622 + 0.684904 i$ \\
                              & $-0.5 + 0.866025 i$       & $-0.856735 + 0.515757 i$ & $-0.0182917 + 0.999833 i$ \\
                              & $-0.856735 - 0.515757 i$  & $0.0182917 + 0.999833 i$ & $0.5 + 0.866025 i$ \\
                              & $1$                       & $1$                      & $1$ \\ \hline
    \end{tabular}
    }
    \caption{Level 3 projection embedding coefficients.}
    \label{tab:lvl-3-proj-coefs}
\end{table}

\section{Equivalences}\label{sec:equivalences}

In this section we prove that the categories $\DD_3$ and $\DD_4$ are indeed
presentations for quantum subgroups of $\GG_2(q_3)$ and $\GG_2(q_4)$, respectively.
It is necessary to note that from \cite{gannon_exotic_q_subgroups_extensions-II} we know that 
the only nontrivial etale algebras in $\ol{\Rep(U_{q_3}(\gg_2))}$ and 
$\ol{\Rep(U_{q_4}(\gg_2))}$ are $A_3$ and $A_4$, respectively.

We first state and prove the theorem at level 4.

\begin{theorem}\label{thm:level-4}
    There is a monoidal equivalence
    \[
        \Ab(\ol{\DD_4}) \cong \ol{ \Rep(U_{q_4}(\gg_2))}_{A_4}.
    \]
\end{theorem}

We first note that $\Ab(\ol{\DD_4})$ is the category of modules for {\it some} etale algebra.

\begin{proposition}
    There is an etale algebra $B_4$ such that 
    \[
        \Ab(\ol{\DD_4}) \cong \ol{ \Rep(U_{q_4}(\gg_2))}_{B_4}.
    \]
\end{proposition}

\begin{proof}
    Recall the $\dagger$-embedding $\ol{\GG_2(q_4)} \hookrightarrow \ol{\DD_4}$ of Corollary~\ref{cor:D4-dagger-emb}.
    From here we take Karoubi completion, which induces a functor 
    \[
        \Ab( \ol{\GG_2(q_4)} ) \hookrightarrow \Ab( \ol{\DD_4} )
    \]
    which is faithful exact.
    By Lemma~\ref{lem:exact-functor}, we deduce that
    \[
        \Ab(\ol{\DD_4}) \cong \Ab( \ol{\GG_2(q_4)} )_{B_4}
    \]
    for some etale algebra object $B_4$.
    But it is well known \cite{Kuperberg} that 
    \[
        \Ab( \ol{\GG_2(q_4)} ) \cong \ol{ \Rep(U_{q_4}(\gg_2))}.
    \]
\end{proof}

Now we are in a position to prove Theorem~\ref{thm:level-4}.

\begin{proof}[Proof of Theorem~\ref{thm:level-4}]
    As we noted before, the only candidates for $B_4$ are $A_4$ and $\unit$.
    Thus it will suffice to demonstrate that 
    \[
        \Ab(\ol{\DD_4}) \not\cong \ol{ \Rep(U_{q_4}(\gg_2))} \cong \ol{ \Rep(U_{q_4}(\gg_2))}_{\unit}.
    \]
    
    The induced functor $\Ab( \ol{\GG_2(q_4)} ) \hookrightarrow \Ab( \ol{\DD_4} )$, whose existence we noted above,
    is defined to be such that the following diagram commutes:
    \[
    \xymatrix@C=60pt@R=45pt{
    \ol{\GG_2(q_4)} \ar@{^{(}->}[r]^{} \ar@{^{(}->}[d]^{} & \ol{\DD_4} \ar@{^{(}->}[d]^{\JJ} \\
    \ol{ \Rep(U_{q_4}(\gg_2))} \ar@{^{(}->}[r]^{} & \Ab( \ol{\DD_4} )
     }
    \]
    By Lemma~\ref{lem:exact-functor} there is an equivalence $\KK$ such that, up to isomorphism, the diagram
    \[
    \xymatrix@C=60pt@R=45pt{
    \ol{\GG_2(q_4)} \ar@{^{(}->}[r]^{} \ar@{^{(}->}[d]^{} & \ol{\DD_4} \ar@{^{(}->}[d]^{\JJ} \\
    \ol{ \Rep(U_{q_4}(\gg_2))} \ar@{^{(}->}[r]^{} \ar[dr]^{\FF_{B_4}} & \Ab( \ol{\DD_4} ) \ar[d]^{\KK} \\
      & \ol{ \Rep(U_{q_4}(\gg_2))}_{B_4} \\
     }
    \]
    commutes.

    Now, chasing the object $1_{ \ol{\GG_2(q_4)}}$ through the leftmost path on this diagram, we have 
    \[
        1_{ \ol{\GG_2(q_4)}} \mapsto V_{\Lambda_1} \mapsto B_4 \otimes V_{\Lambda_1}.
    \]
    If it was true that $B_4\cong \unit$, we would have 
    \[
        \dim \Hom_{\ol{ \Rep(U_{q_4}(\gg_2))}_{B_4}}((B_4 \otimes V_{\Lambda_1})^{\otimes 2} \to B_4 \otimes V_{\Lambda_1}) = 1.
    \]
    Now let us chase the object $1$ along the top rightmost path to demonstrate that this is not the case.
    Along this second path we have
    \[
        1_{ \ol{\GG_2(q_4)}} \mapsto 1_{\ol{\DD_4}} \mapsto \JJ(1_{\ol{\DD_4}}) \mapsto \KK\circ\JJ(1_{\ol{\DD_4}}).
    \]
    Set $\CC_4 \coloneq \ol{ \Rep(U_{q_4}(\gg_2))}$.
    Now we note that
    \begin{align*}
        \dim \Hom_{(\CC_4)_{B_4}}((B_4 \otimes V_{\Lambda_1})^{\otimes 2} \to B_4 \otimes V_{\Lambda_1}) & = 
            \dim \Hom_{(\CC_4)_{B_4}}((\KK\circ\JJ(1_{\ol{\DD_4}}))^{\otimes 2} \to \KK\circ\JJ(1_{\ol{\DD_4}})) \\
        & = \dim \Hom_{\Ab( \ol{\DD_4} )}((\JJ(1_{\ol{\DD_4}}))^{\otimes 2} \to \JJ(1_{\ol{\DD_4}})) \\
        & = \dim \Hom_{\ol{\DD_4} }( (1_{\ol{\DD_4}})^{\otimes 2} \to 1_{\ol{\DD_4}}) \\
        & = \dim \Hom_{\ol{\DD_4} }( 2_{\ol{\DD_4}} \to 1_{\ol{\DD_4}}) \\
        & = 2.
    \end{align*}
    Since $B_4 \not\cong\unit$, we must have $B_4 \cong A_4$; the theorem is proved.
\end{proof}

The category $\DD_3$, whose $\Z_3$-like structure constants are given in the Mathematica notebook
\verb|/code/level-3/nb-6_verify-g.nb|, admits an analogous theorem at level 3.

\begin{theorem}\label{thm:level-3}
    There is a monoidal equivalence
    \[
        \Ab(\ol{\DD_3}) \cong \ol{ \Rep(U_{q_3}(\gg_2))}_{A_3}.
    \]
\end{theorem}

The proof of Theorem~\ref{thm:level-3} is analogous to the argument given above.

\pagebreak

\printbibliography

\end{document}